\documentclass[graybox]{svmult}

\usepackage{amssymb}

\usepackage{mathptmx}       
\usepackage{helvet}         
\usepackage{courier}        
\usepackage{type1cm}        
\usepackage{makeidx}         
\usepackage{graphicx}        
\usepackage{multicol}        
\usepackage[bottom]{footmisc}

\begin{document}

\title*{Empiric stochastic stability of physical and pseudo-physical measures.}
\author{Eleonora  Catsigeras\\ {\ \ } \\ {\em Dedicated to the memory of Prof. Welington de Melo}}
  \authorrunning{Eleonora  Catsigeras}
\institute{Instituto de Matem\'{a}tica y Estad\'{\i}stica \lq\lq Rafael Laguardia\rq\rq, Universidad de la Rep\'{u}blica. \at Av. J. Herrera y Reissig 565, Montevideo, URUGUAY. \\ \email{eleonora@fing.edu.uy}}

\maketitle

\abstract*{\texttt We define the empiric  stochastic stability of an invariant  measure   in the finite-time scenario, the classical definition of stochastic stability.  We prove that an invariant measure of a  continuous system  is empirically stochastically stable if and only if it is physical. We also define the empiric stochastic stability of a weak$^*$-compact set   of invariant measures  instead of a single measure.  Even when the system  has not  physical measures   it still has   minimal empirically stochastically stable sets  of measures.  We prove that such sets are necessarily composed by pseudo-physical measures. Finally, we apply the results to the one-dimensional C1-expanding case  to conclude that the measures of empirically stochastically   sets  satisfy Pesin Entropy Formula.}

\abstract{. We define the empiric  stochastic stability of an invariant  measure   in the finite-time scenario, the classical definition of stochastic stability.  We prove that an invariant measure of a  continuous system  is empirically stochastically stable if and only if it is physical. We also define the empiric stochastic stability of a weak$^*$-compact set   of invariant measures  instead of a single measure.  Even when the system  has not  physical measures   it still has   minimal empirically stochastically stable sets  of measures.  We prove that such sets are necessarily composed by pseudo-physical measures. Finally, we apply the results to the one-dimensional C1-expanding case  to conclude that the measures of empirically stochastically   sets  satisfy Pesin Entropy Formula.}

\section{Introduction}

 The purpose of this paper is to study a type of stochastic stability of invariant measures, which we   call \lq\lq empiric stochastic stability\rq\rq   for   continuous maps  $f: M \mapsto M$   on a compact Riemannian manifold $M$ of finite dimension, with or without boundary.  In particular, we are interested on the empirically stochastically stable measures of  one-dimensional continuous dynamical systems, and among them, the $C^1$-expanding maps on the circle.

 Let us denote by $(M, f)$ the deterministic (zero-noise) dynamical system obtained  by iteration of $f$, and by $(M, f, P_{\epsilon})$ the randomly perturbed system whose noise amplitude is $\epsilon$.
 Even if we  will work on a wide scenario which includes any continuous dynamical system   $(M, f)$,  we restrict the stochastic system  $(M, f, P_{\epsilon})$  by assuming that the noise probability distribution is uniform (i.e. it has constant density) on  all the  balls of radius $\epsilon >0$ of $M$  (for a precise statement of this assumption see formula (\ref{equationp_epsilon}) below).
   We call $\epsilon$ the \em noise level, \em or also the \em   amplitude of the random perturbation\em.
   To define the empiric stochastic stability we will take $\epsilon \rightarrow 0^+ $.

   In the stochastic system $(M, f, P_{\epsilon})$, the symbol $P_{\epsilon}$ denotes the family of  probability distributions, which  are called   \em transition probabilities, \em according to which the noise  is added to $f(x)$ for each    $x \in M$. Precisely, each transition probability is, for all $n \in \mathbb{N}$,  the  distribution  of the state $x_{n+1}$ of the noisy orbit conditioned to   $x_n = x$,   for each $x \in M$. As said above, the transition probability is supported on the ball with center at $f(x)$ and radius $\epsilon>0$. So, the zero-noise system $(M,f)$ is recovered  by taking $\epsilon= 0$; namely, $(M, f) =(M, f, P_0)$. The observer naturally expects that  if the amplitude $\epsilon >0$ of the random perturbation were small enough, then the ergodic properties of the stochastic system   \lq\lq remembered\rq\rq \ those of the zero-noise system.

  The foundation and tools to study the random perturbations of dynamical systems were early provided   in \cite{0-Ulam-vonNeumann1947}, \cite{Arnold}, \cite{Kifer1974Article}. The stochastic stability appears in the literature  mostly defined  through the stationary meaures $\mu_{\epsilon}$ of the stochastic system $(M, f, P_{\epsilon})$  Classically, the authors prove and describe, under particular  conditions, the existence and properties of the $f$-invariant measures that are the weak$^*$-limit of  ergodic stationary measures as  $\epsilon \rightarrow 0^+$.  See for instance the early results of \cite{Young}, \cite{Kifer1986}, \cite{Brin-Kifer}, \cite{Kifer1988}, \cite{Kifer1990Article-7}), and  the later  works of \cite{11-Liu-Qian}, \cite{8-AraujoTesis}, \cite{9-Araujo-Alves-Vasquez}, \cite{9-Araujo-Tahzibi}. For a review on stochastic and statistical stability of  randomly perturbed dynamical systems, see for instance \cite{12-Viana} and Appendix D of \cite{Bonatti-Diaz-Viana}.

  The stationary measures of the ramdom perturbations provide the probabilistic behaviour of the noisy system  asymptotically in the future.
Nevertheless,  from a   rather practical or experimental point of view  the concept of stochastic stability should not require the knowledge  a priori of the limit measures   of the perturbed system as $n \rightarrow + \infty$ .   For instance  \cite{Pierre-Sandro-numericalExperiment} presents numerical    experiments on the  stability of one-dimensional noisy systems in a finite time. The ergodic stationary measure is in fact substituted by an empirical (i.e. obtained after a finite-time observation of the system)  probability. Also in other applications of the theory of random systems  (see for instance \cite{50-Friedman1975}, \cite{51-Ikeda-Watanabe1981}), the stationary measures are usually unkown, are not directly obtained from the experiments, but substituted by  the   finite-time empiric probabilities which approximate  the stationary measures if the observations last   enough.

Summarizing, for a certain type of stochastically stable properties, one should not need   the infinite-time noisy orbits. Instead, one may   take   the noisy orbits \em up to a   large  finite  time $n$, which  \em  are indeed those that the experimenter observes and  predicts. The statistics of the observations and predictions   of the noisy orbits still reflect, for the experimenter and  the predictor, the behaviour   of the stochastic system, \em but only up to  some finite horizon. \em

Motivated by the above arguments, in Section \ref{sectionDefinitions} we will define \em the empiric stochastic stability. \em Roughly speaking, an $f$-invariant probability for the zero-noise system $(M, f)$ is empirically stochastically stable  if it approximates, up to an arbitrarily small error $\rho>0$, the statistics of sufficiently large pieces  of the noisy orbits, for  some fixed time $n$, provided that the noise-level $\epsilon>0$ is small enough  (see Definition \ref{definitionEmpiricStochasticStabilityMeasure}). This concept is a  reformulation in a \em finite-time \em scenario of one of  the usual definition   of  infinite-time  stochastic stability (see for instance \cite{Young}, \cite{Brin-Kifer}, \cite{9-Araujo-Alves-Vasquez}).

 \subsection{Setting the problem}

 Let $\epsilon > 0$ and $x \in M$. Denote by $B_{\epsilon}(x) \subset M$  the open ball of radius $\epsilon$ centered at $x$.  Consider  the Lebesgue measure $m$, i.e. the finite measure obtained from the volume form induced by the Riemannian structure of the manifold. For each point $x \in M$, we take the restriction of $m$ to the ball $B_ \epsilon (f(x))$. Precisely, we define  the probability measure $p_{\epsilon} (x, \cdot)$ by the following equality:
 \begin{equation} \label{equationp_epsilon} p_{\epsilon}(x, A) := \frac{m \big(A \cap B_{\epsilon}(f(x))\big)}{m\big(B_{\epsilon}(f(x))\big)} \ \ \forall \ A \in {\mathcal A},\end{equation}
 where ${\mathcal A}$ is the Borel sigma-algebra in $M$.

 \begin{definition} {\bf (Stochastic system with noise-level $\epsilon$.)}
  \label{definitionStochasticSystem}

  For each value of $\epsilon > 0$, consider the stochastic process or Markov chain $\{x_n\}_{n \in  \mathbb{N}} \subset M^{\mathbb{N}}$ in the   measurable space ${(M, {\mathcal A})}$ such that, for all $A \in {\mathcal A}$:
  $$\mbox{prob}(x_0 \in A) = m(A), \ \ \mbox{prob}(x_{n+1} \in A | x_n= x) = p_{\epsilon}(x, A),$$
  where $p_{\epsilon}(x, \cdot)$ is defined by equality (\ref{equationp_epsilon}).

  The system whose stochastic orbits are the Markov chains as above  is called \em stochastic system with noise-level $\epsilon$. \em   We denote it by $(M, f, P_{\epsilon})$, where
   $$P_{\epsilon} := \{p_{\epsilon}(x, \cdot)\}_{x \in M}.$$

  \end{definition}

  The stochastic systems with noise-level $\epsilon>0$ are usually studied  by assuming certain regularity of the zero-noise systems $(M, f)$,  and by taking the ergodic stationary measures $\mu_{\epsilon}$ of the stochastic system $(M, f, P_{\epsilon})$ (see for instance \cite{Young}).  When assuming that the transition probabilities satisfy  equality (\ref{equationp_epsilon}), all the stationary probability measures become absolutely continuous with respect to the Lebesgue measure $m$ (see for instance \cite{Vaienti}). Therefore, if a property   holds for the noisy orbits for $\mu_{\epsilon}$- a.e initial state $x \in M$ , it also holds  for a Lebesgue-positive set of states.

  When looking at the noisy system,  the experimenter usually obtains the values of  several bounded measurable functions $\varphi$, which are called \em observables\em, along the stochastic orbits $\{x_n\}_{n \in \mathbb{N}}$. From Definition \ref{definitionStochasticSystem}, the expected value of $\varphi$ at instant 0  is $E(\varphi)_0= \int \varphi(x_0) \, dm (x_0) $. Besides, from the   definition of the transition probabilities by equality (\ref{equationp_epsilon}), for any given state $x \in M$  the expected value  of $\varphi(x_{n+1})$  conditioned to $x_n = x$  is
 $   \int \varphi(y) \,  p_{\epsilon}(x, dy).$ So, in particular at instant 1  the expected value of $\varphi$  is
 $$E(\varphi)_1= \int \!\!\!\int \varphi(x_1)\,   p_{\epsilon}(x_0, dx_1) \, dm(x_0),$$
 and its expected value at instant $2$ is
 $$ E(\varphi)_2= \int \! \! \!\int \! \!\! \int  \varphi(x_2)  \, p_{\epsilon}(x_1, dx_2)\, p_{\epsilon}(x_0, dx_1) \, dm(x_0).$$ Analogously, by induction on $n$  we obtain that  for all $n \geq 1$, the expected value $E(\varphi)_n $ of the observable $\varphi$ is
 \begin{equation} \label{equationExpectedValue} E(\varphi)_n= \int \! \! \!\int \! \!\! \int ...\int \varphi(x_n) p_{\epsilon}(x_{n-1}, dx_{n})...\, p_{\epsilon}(x_1, dx_2)\, p_{\epsilon}(x_0, dx_1) \, dm(x_0).\end{equation}
Since the Lebesgue measure $m$ is not necessarily stationary for the   system $(M, f, P_{\epsilon})$, the expected value of the same function $\varphi$ at each instant $n$, if the initial distribution is $m$, may change with   $n$.

As said at the beginning, we assume that the experimenter  only sees the values of the observable functions  along   finite pieces  of  the noisy orbits because his experiment and his empiric observations can not last forever. When analyzing the statistics of the observed data, he considers for instance   the time average  of the collected observations along those finitely elapsed pieces of  randomly perturbed orbits. These time averages can be computed by the integrals of the observable functions with respect to certain probability measures, which are called    \em empiric stochastic probabilities for finite time $n$ \em (see Definition \ref{definitionEmpiricStochasticProb}). Precisely, for any any fixed time $n \geq 1$ and for any initial state $x_0 \in M$,   the empiric stochastic probability $\sigma_{\epsilon, n, x_0}$ is defined such that the  time average   of the expected values of any observable $\varphi$ at instants $1, 2, \ldots, n$  along the noisy orbit initiating at $x_0$, can be computed by the following equality:
$$  \frac{1}{n} \sum_{j= 1}^n E(\varphi(x_j)| {x_0}) = \int \varphi (y) d \sigma_{\epsilon, n, x_0}(y),$$
where
\begin{equation} \label{equationExpectedValueGivenx_0}E(\varphi(x_j)| {x_0}) = \int \!\!\! \int \ldots \int \varphi(x_j) \, p_{\epsilon}(x_{j-1}, dx_j) \ldots p_{\epsilon} (x_1, dx_2) p_{\epsilon} (x_0, dx_1).\end{equation}

We also assume that the experimenter only sees Lebesgue-positive sets  in the phase space $M$. So, when analyzing the statistics of the observed data in the noisy system, he will not observe all the empiric stochastic distributions $\sigma_{\epsilon, n, x}$, but only those for  Lebesgue-positive sets of initial states $x \in M$. If besides  he can only manage a finite set of continuous observable functions, then he will  not see the exact probability distributions, but some  weak$^*$ approximations to them   up to an error $\rho>0$, in the metric space ${\mathcal M}$ of probability measures.



 For some classes of mappings on the manifold $M$, even with high regularity (for instance Morse-Smale $C^{\infty}$ diffeomorphisms with two or more hyperbolic sinks), one single measure $\mu $ is not enough to approximate the empiric stochastic probabilities  of the noisy orbits  for  Lebesgue-a.e.   $x \in M$. The experimenter may need a set ${\mathcal K} $ composed by several probability measures  instead of a single measure.  Motivated by this phenomenon, we define the \em empiric stochastic stability  of  a weak$^*$-compact set  ${\mathcal K}$ \em of $f$-invariant probability measures (see Definition \ref{definitionEmpiricStochasticStabilitySetK}). This concept is similar to the empiric stochastic stability of a single measure, with two main changes: first, it substitutes the   measure $\mu$ by a weak$^*$-compact set $\mathcal K$ of probabilities; and second, it requires $\mathcal K$ be minimal with the property of empiric stochastic stability, when restricting  the stochastic system to a fixed Lebesgue-positive set of noisy orbits. In particular, a \em globally \em empirically stochastically stable set ${\mathcal K}$ of invariant measures minimally approximates the statistics of Lebesgue-a.e. noisy orbits. We will prove that it exists and is unique. 
 

\subsection{Main results.}

A classical concept in the ergodic theory of zero-noise dynamical systems is that of \em physical measures \em  \cite{14-PhysicalMeasures-Eckmann-Ruelle-1985}.   In brief, a physical measure is an $f$-invariant measure  $\mu$ whose basin of statistical attraction has positive Lebesgue measure. This basin is composed by the zero-noise orbits such that the  time average probability up to time $n$  converges to $\mu$ in the weak$^*$-topology as  $n \rightarrow + \infty$ (see Definitions \ref{definitionBasinStatisticalAttraction} and \ref{definitionPhysicalMeasures}).

One of the main purposes of this paper is to answer the following question:

\vspace{.3cm}

\noindent {\bf Question 1.}
 Is there some relation between the empirically stochastically stable   measures and the physical  measures?
If yes, how are they related?

\vspace{.1cm}

We will give an answer to this question in Theorem \ref{theoremEmpiricStochasticStableMeasure=Physical} and Corollary \ref{corollaryGlobalStabilityUniquePhysicalMeasure} (see Subsection \ref{subsectionStatementOfResults} for their precise statements). 
In particular, we will prove the following result:

\vspace{.3cm}

\noindent {\bf Theorem. }  
\em  An $f$-invariant measure is empirically stochastically stable if and only if it is physical. \em 

\vspace{.3cm}

A generalization of physical measures,  is the concept of \em pseudo-physical probability measures, \em which are sometimes also called SRB-like measures \cite{Polaca}, \cite{Portugalia}, \cite{Cat-StatisticalAttractors}.   They are defined such that, for all $\rho >0$, their weak$^*$ $\rho$-neighborhood, has a (weak) basin of statistical attraction with positive Lebesgue measure  (see Definitions \ref{definitionBasinStatisticalAttraction} and \ref{definitionPhysicalMeasures}).

To study this more general scenario of pseudo-physics, our second main purpose is to answer the following question:

\vspace{.3cm}

\noindent {\bf Question 2. }
Do empirically stochastically stable sets of measures relate with pseudo-physical measures?
If yes, how do they relate?

\vspace{.1cm}

We will give an answer to this question in Theorem \ref{theoremEmpricStochasticStableSetsSubsetPseudoPhysical}
and its corollaries, whose precise statements are in   Subsection \ref{subsectionStatementOfResults}.  
In particular, we will prove the following result:

\vspace{.3cm}

\noindent {\bf Theorem. }  
\em A weak$^*$-compact set of invariant probability measures is empirically stochastically stable  only if all its measures are pseudo-physical. Conversely, any pseudo-physical measure belongs to the unique globally empirically stochastically stable set of measures. \em

\newpage


\section{Definitions and statements} \label{sectionDefinitions} 

We denote by ${\mathcal M}$ the space of Borel probability measures on the manifold $M$, endowed with the weak$^*$-topology; and by ${\mathcal M}_f$ the subspace of $f$-invariant probabilities, where $(M, f)$ is the zero-noise dynamical system.
 Since the weak$^*$ topology in $\mathcal M$ is metrizable, we can choose and fix a metric $\mbox{dist}^*$ that endows that topology.


To make formula (\ref{equationExpectedValue}) and  other computations concise, it is convenient to introduce the following definition:

 \begin{definition}
 {\bf The transfer operators ${\mathcal L}_{\epsilon}$ and ${\mathcal L}^*_{\epsilon}$.}
 \label{definitioTransferOperators}

 Denote by $C^0(M, \mathbb{C})$ the space of complex continuous functions defined in $M$.  For the stochastic system $(M, f, P_{\epsilon})$, we define the \em transfer operator \em ${\mathcal L}_{\epsilon}: C^0(M, \mathbb{C}) \mapsto C^0(M, \mathbb{C})$  as follows:
 \begin{equation}
 \label{equationL} ({\mathcal L}_{\epsilon} \varphi) (x) := \int \varphi_(y) \, p_{\epsilon}(x, dy) \ \ \forall \ x \in M, \ \ \forall \ \varphi \in C^0(M, \mathbb{C}).\end{equation}

 From equality (\ref{equationp_epsilon}) it is easy to prove that $p_{\epsilon}(x, \cdot)$ depends continuously on $x \in M$ in the weak$^*$ topology. So, ${\mathcal L}_{\epsilon} \varphi$ is a continuous function for any $\varphi \in C^0(M, \mathbb{C})$.

 Through Riesz representation theorem, for any measure $\mu \in {\mathcal M}$  there exists a unique measure, which we denote by  $  {\mathcal L}^*_{\epsilon} \mu$, such that
 \begin{equation}
 \label{equationL^*}
 \int \varphi d({\mathcal L}_{\epsilon}^* \mu) := \int ({\mathcal L}_{\epsilon} \varphi)\, d \mu \ \ \ \forall \ \varphi \in C^0(M, \mathbb{C}).
 \end{equation}
 We call ${\mathcal L}_{\epsilon}^*:{\mathcal M} \mapsto {\mathcal M}$ \em the dual transfer operator \em or also,  \em the transfer operator in the space of  measures. \em
 \end{definition}
From the above definition, we obtain the following property for any \em observable function\em \ $\varphi \in C^0(M, \mathbb{C}) $: its expected value at the instant $n$ along the stochastic orbits with noise level $\epsilon$ is
$$E(\varphi)_n = \int ({{\mathcal L}_{\epsilon}}^n \varphi) \, d m = \int  \varphi  \, d({{\mathcal L}^*_{\epsilon}}^n m). $$

We are not only interested in the expected values of the observables $\varphi$, but also \em in the statistics \em (i.e time averages of the observables) along the  individual noisy orbits. 
With such a purpose, we first consider the following equality:
\begin{equation}
\label{equationL*delta_x}
({{\mathcal L}_{\epsilon}} ^n \varphi) (x) =  \int \varphi  \, d({{\mathcal L}^*_{\epsilon}}^n \delta_x)  \ \ \ \forall \ x \in M,\end{equation}
where $\delta_x$ denotes the Dirac probability measure supported on $\{x\}$.
Second,   we introduce the following concept of empiric probabilities for the stochastic system:

\begin{definition}{\bf Empiric stochastic probabilities.} \label{definitionEmpiricStochasticProb}

For any fixed instant $n \geq 1$, and for any initial state $x \in M$, we define the \em empiric stochastic probability $\sigma_{\epsilon, n,x} $ \em of the noisy orbit  with noise-level $\epsilon>0$, with initial state  $x$, and up to time $n$, as follows:
\begin{equation}
\label{equationEmpiricStochasticProb}
\sigma_{\epsilon, n,x} := \frac{1}{n} \sum_{j= 1}^{n} {{\mathcal L}_{\epsilon}^*}^j \delta_x.
\end{equation}
\end{definition}

Note that the empiric stochastic probabilities for Lebesgue almost $x \in M$   allow the computation of the time averages of the observable  $\varphi$ along the noisy   orbits.   Precisely,
\begin{equation} \label{equationInt_sigma_epsilon,n,x}\frac{1}{n} \sum_{j= 1}^{n} ({\mathcal L}_{\epsilon}^j \varphi) (x) = \int \varphi (y) \, d \sigma_{\epsilon, n,x}(y) \  \ \ \forall \ \varphi \in C^0(M, \mathbb{C}).\end{equation}

\begin{definition} \label{definitionEmpiricStochasticStabilityMeasure}
{\bf (Empiric stochastic stability of  a  measure)}

We call a probability measure $\mu \in {\mathcal M}_f$    \em empirically stochastically stable \em if there exists a measurable set $\widehat A  \subset M$ with positive Lebesgue measure such that:

For all $\rho > 0$ and for all $n \in  \mathbb{N}^+$ large enough   there exists $\epsilon_0>0$ (which  may depend on $\rho$ and on $n$ but not on $x$) satisfying
$$\mbox{dist}^*(\sigma_{\epsilon, n, x}, \ \mu) < \rho \  \ \ \ \forall \ 0 <\epsilon \leq \epsilon_0 , \mbox{ for Lebesgue a.e. }   x \in \widehat A.$$

\end{definition}
\begin{definition}
\label{definitionBasinWidehatA_mu} {\bf (Basin of empiric stochastic stability of a measure)}

For any probability measure $\mu$, we construct the following (maybe empty) set in the ambient manifold $M$:
\begin{equation}
\label{equation01}
 \widehat A_{\mu} :=\Big \{x \in M \colon \ \ \forall \rho >0 \ \exists \ N= N(\rho) \mbox{ such that } \forall \ n \geq N \ \exists \ \epsilon_0 = \epsilon_0(\rho, n) >0 \mbox { satisfying } $$ $$ \mbox{dist}^*(\sigma_{\epsilon, n, x}, \ \mu) < \rho \  \ \ \ \forall \ 0 <\epsilon \leq \epsilon_0 \Big \}.\end{equation}

We call the set $\widehat A_{\mu} \subset M$ \em the basin of empiric stochastic stability of $\mu$. \em Note that it is defined for any probability measure $\mu \in {\mathcal M}$, but it may be empty, or even if nonempty, it may have zero Lebesgue-measure when $\mu$ is not empirically stochastically stable.

The set $\widehat A_{\mu}$ is measurable (see Lemma \ref{lemma3}). According to Definition \ref{definitionEmpiricStochasticStabilityMeasure}, a probability measure  $\mu$ is empirically stochastically stable  if and only if the set $\widehat A_{\mu}$ has positive Lebesgue measure (see Lemma \ref{lemma4}).
\end{definition}
\begin{definition}
\label{definitionGloballyEmpStochStabilityMeasure} {\bf (Global empiric stochastic stability of a measure)}

We say that $\mu \in {\mathcal M}_f$ is \em globally empirically stochastically stable \em if  it is empirically stochastically stable, and besides its basin  $\widehat A_{\mu}$   of empiric stability has full Lebesgue measure.

\end{definition}
\begin{definition}
\label{definitionBasinWidehatA_K} {\bf (Basin of empiric stochastic stability of a set of measures)}

For any nonempty weak$^*$-compact set $\mathcal K \subset {\mathcal M}$, we construct the following (maybe empty) set in the space manifold $M$:
\begin{equation} \label{equation11} \widehat A_{\mathcal K} := \{x \in M \colon \ \ \forall \rho >0 \ \exists \ N= N (\rho) \mbox{ such that } \forall \ n \geq N \ \exists \ \epsilon_0 = \epsilon_0(\rho, n) >0 \mbox { satisfying } $$ $$ \mbox{dist}^*(\sigma_{\epsilon, n, x}, \ \mathcal K) < \rho     \ \ \ \forall \ 0 <\epsilon \leq \epsilon_0  \}.\end{equation}

We call $\widehat A_{\mathcal K} \subset M$ \em the basin of empiric stochastic stability of $\mathcal K$. \em

\end{definition}
Note that $\widehat A_{\mathcal K} $ is defined for any nonempty weak$^*$-compact set ${\mathcal K} \subset {\mathcal M}$. But it may be empty, or even if nonempty,  it may have zero Lebesgue measure when ${\mathcal K}$ is not empirically stochastically stable, according to the following definition:

\begin{definition} \label{definitionEmpiricStochasticStabilitySetK}
{\bf (Empiric stochastic stability of  a set of measures)}

We call a nonempty weak$^*$-compact set ${\mathcal K} \subset {\mathcal M}_f$ of $f$-invariant probability measures     \em empirically stochastically stable \em if :

\begin{itemize}
\item[a) ] There exists a measurable set $\widehat A  \subset M$ with positive Lebesgue measure, such that:

For all $\rho > 0$ and for all $n \in  \mathbb{N}^+$ large enough, there exists $\epsilon_0>0$ (which  may depend on $\rho$ and  $n$, but not on $x$), satisfying:
$$\mbox{dist}^*(\sigma_{\epsilon, n, x}, \ {\mathcal K}) < \rho \  \ \ \ \forall \ 0 <\epsilon \leq \epsilon_0, \ \ \ \forall \ x \in \widehat A.$$
\item[b) ]  ${\mathcal K}$ is   minimal in the following sense:  if   ${\mathcal K}' \subset {\mathcal M}_f$ is  nonempty and weak$^*$-compact, and if   ${\widehat A}_{\mathcal K} \subset {\widehat A}_{{\mathcal K}'}$ Lebesgue-a.e., then ${\mathcal K } \subset {\mathcal K}'$.
\end{itemize}

By definition,  if $\mathcal K$ is empirically stochastically stable, then the set $\widehat A \subset M$ satisfying condition a),  has positive Lebesgue measure and is contained in $\widehat A_{\mathcal K}$. Since  $\widehat A_{\mathcal K}$ is measurable (see Lemma \ref{lemma3bis}), we conclude that it has positive Lebesgue measure.

Nevertheless,  for a nonempty weak$^*$-compact set ${\mathcal K}$ be empirically stochastically stable, it is not enough that  $\widehat A_{\mathcal K} $ has positive Lebesgue measure. In fact, to avoid the whole set ${\mathcal M}_f$ of $f$-invariant measures be always an empirically stochastically stable set, we ask ${\mathcal K}$ to satisfy   condition b). In brief, we require  a property of minimality of ${\mathcal K}$ with respect to  Lebesgue-a.e. point  of its basin $\widehat A_{\mathcal K}$ of empiric stochastic stability.
\end{definition}
\begin{definition}
\label{definitionGloballyEmpStochStabilitySet} {\bf (Global empiric stochastic stability of a set of measures)}

We say that a nonempty weak$^*$-compact set  ${\mathcal K} \in {\mathcal M}_f$ is \em globally empirically stochastically stable \em if  it is empirically stochastically stable, and besides its basin  $\widehat A_{\mathcal K}$ of empiric stability  has full Lebesgue measure.

\end{definition}

  We recall the following definitions from \cite{Polaca}:

\begin{definition}
{\bf (Empiric zero-noise probabilities and $p\omega$-limit sets)}
\label{definitionEmpiricProbab&pomega}

For any fixed natural  number $n \geq 1$,   the \em empiric probability $\sigma_{n, x} $ \em of the   orbit with initial state  $x \in M$  and up to time $n$  of the zero-noise system $(M, f)$,  is defined by the following equality:
$$\sigma_{ n, x} := \frac{1}{n} \sum_{j= 1}^n \delta_{f^j(x)}.$$
It is standard to check, from the construction of the empiric stochastic probabilities in Definition \ref{definitionEmpiricStochasticProb}, that   $\sigma_{\epsilon, n, x}$  is absolutely continuous with respect to the Lebesgue measure $m$. In contrast,  the empiric probability $\sigma_{n,x}$ for the zero-noise orbits is \em atomic\em, since  it  is supported on a finite number of points.

\vspace{.2cm}

The \em p-omega limit set $p\omega_x$ \em in the space ${\mathcal M}$ of probability measures, corresponding to the orbit of $x \in M$, is defined by:
$$p\omega_x := \{\mu \in {\mathcal M}: \ \exists \ n_i \rightarrow + \infty \mbox{ such that } {\lim}^*_{i \rightarrow + \infty} \sigma_{n_i, x} = \mu\}, $$
where ${\lim}^*$ is taken in the weak$^*$-topology of ${\mathcal M}$.
 It is standard to check that $p \omega_x \subset {\mathcal M}_f$ for all $x \in M$.
\end{definition}
\begin{definition}
\label{definitionBasinStatisticalAttraction} {\bf (Strong and $\rho$-weak basin of statistical attraction)}

For any $f$-invariant probability measure $\mu \in {\mathcal M}_f$,   the (strong) \em basin   of statistical attraction of $\mu$  \em  is the (maybe empty) set
\begin{equation}
\label{equation02}
A_{\mu} \colon = \big \{x \in M \colon \ \ p\omega_x = \{\mu\} \big \}.\end{equation}
 For any $f$-invariant probability measure $\mu \in {\mathcal M}_f$,   and for any $\rho >0$, the   \em $\rho$-weak basin of statistical attraction of $\mu$  \em  is the (maybe empty) set
$$A^{\rho}_{\mu} \colon = \big \{x \in M \colon \ \ \mbox{dist}^*(p\omega_x , \{\mu\}) < \rho \big \}.$$
\end{definition}

\begin{definition}
\label{definitionPhysicalMeasures} {\bf (Physical and pseudo-physical measures)}

For the zero-noise dynamical system $(M, f)$, an $f$-invariant probability measure $\mu $ is   \em physical \em if
its strong basin of statistical attraction $A_{\mu}$ has positive Lebesgue measure.

An $f$-invariant probability measure $\mu$ is \em pseudo-physical \em if for all $\rho >0$, its $\rho$-weak basin of statistical attraction $A^{\rho}_{\mu}$ has positive Lebesgue measure.
\end{definition}
It is standard to check that, even if the $\rho$-weak basin of statistical attraction $A^{\rho}_{\mu}$ depends on the chosen weak$^*$ metric in the space ${\mathcal M}$ of probabilities, the set of pseudo-physical measures remains the same when changing this metric (provided that the new metric also endows the weak$^*$-topology).

Note that the strong basin of statistical attraction of any measure is always contained in the $\rho$-weak basin of the same measure. Hence, any physical measure (if there exists some) is pseudo-physical. But not all the pseudo-physical measures are necessarily physical (see for instance example 5 of \cite{Cat-StatisticalAttractors}).

We remark that we \em do not require  the ergodicity  of $\mu$   \em to be physical or pseudo-physical. In fact, in  \cite{Golenischeva}  it is proved that
the $C^{\infty}$ diffeomorphism, popularly known as the Bowen Eye, exhibits a segment of pseudo-physical measures whose extremes, and so all the measures in the segement, are non ergodic. Also, for some $C^0$-version of Bowen Eye (see example 5 B of \cite {Cat-StatisticalAttractors}) there is a unique pseudo-physical measure, it is physical and non-ergodic.

\newpage

\subsection{Statement of the results} \label{subsectionStatementOfResults}

\begin{theorem}
{\bf (Characterization of empirically stochastically stable measures)}
\label{theoremEmpiricStochasticStableMeasure=Physical}

Let $f: M \mapsto M$ be a continuous map  on a compact Riemannian manifold $M$. Let $\mu$ be an $f$-invariant probability measure. Then, $\mu$ is empirically stochastically stable if and only if it is physical.

Besides, if $\mu$ is physical, then its   basin $\widehat A_{\mu} \subset M  $ of empiric stochastic stability equals Lebesgue-a.e. its strong basin $A_{\mu} \subset M$ of  statistical attraction.
\end{theorem}

 We will prove Theorem \ref{theoremEmpiricStochasticStableMeasure=Physical} and the following corollaries in Section \ref{sectionProofTheorem1}.

\begin{corollary}
\label{corollaryGlobalStabilityUniquePhysicalMeasure}
Let $f: M \mapsto M$ be a continuous map  on a compact Riemannian manifold $M$. Then, the following conditions are equivalent:

\begin{itemize}
\item[(i)  ]   \ \ \ There exists an $f$-invariant probability measure $\mu_1$ that is globally empirically stochastically stable.

\item[(ii)  ]     \ \ \ There exists an $f$-invariant probability measure $\mu_2$ that is physical and such that its strong basin of statistical attraction has full Lebesgue measure.

\item[(iii)  ]   \ \ \  There exists a unique $f$-invariant probability measure $\mu_3$ that is pseudo-physical.
\end{itemize}
Besides, if (i), (ii) or (iii) holds, then $\mu_1= \mu_2 = \mu_3$, this measure is the unique empirically stochastically stable, and the set $\{\mu_1\}$ is the unique weak$^*$-compact set in the space of probability measures that is empirically stochastically stable.
\end{corollary}

Before stating the next corollary, we fix the following definition: we say that a property of  the maps  on $M$  is \em $C^1$-generic \em if it holds for  a countable intersection of open and dense sets of maps in the $C^1$- topology. 

\begin{corollary}
\label{corollaryC1&C2ExpandingMapsInTheCircle}

For $C^1$-generic and for all $C^2$ expanding maps of the circle, there exists a unique ergodic  measure $\mu$ that is empirically stochastically stable. Besides $\mu$ is globally empirically stochastically stable and it is the unique measure that satisfies the following  Pesin Entropy Formula \em \cite{EntropyFormula-Ledrappier-Young}, \cite{EntropyFormula-Liu-Qian}\em:
\begin{equation}
\label{eqnPesinEntropyFormula}
h_{\mu}(f) = \int \log |f'| \, d \mu.\end{equation}

\end{corollary}

Theorem \ref{theoremEmpiricStochasticStableMeasure=Physical} is a particular case of the following  result:

\begin{theorem}
{\bf (Empirically stochastically stable sets and pseudo-physics)}
\label{theoremEmpricStochasticStableSetsSubsetPseudoPhysical}

Let $f: M \mapsto M$ be a continuous map  on a compact Riemannian manifold $M$.

\noindent{(a) } If ${\mathcal K}$ is a nonempty weak$^*$-compact set of $f$-invariant measures that is empirically stochastically stable, then any $\mu \in {\mathcal K}$ is pseudo-physical.

\noindent{(b) } A set ${\mathcal K}$  of $f$-invariant measures is globally empirically stochastically stable  if and only if it coincides with the set of all the pseudo-physical measures.

\end{theorem}

We will prove Theorem \ref{theoremEmpricStochasticStableSetsSubsetPseudoPhysical} and the following corollaries in Section \ref{sectionProofTheorem2}.

 \begin{corollary}   \label{corollary2}
 For any continuous map  $f: M \mapsto M$   on a compact Riemannian manifold $M$, there exists and is unique the nonempty weak$^*$-compact set ${\mathcal K}$   of $f$-invariant measures that is globally stochastically stable. Besides, $\mu \in {\mathcal K}$ if and only if $\mu$ is pseudo-physical.
 \end{corollary}
\begin{corollary} \label{corollary3}
 If a pseudo-physical measure $\mu$  is isolated in the set of  pseudo-physical measures, then it is empirically stochastically stable; hence physical.
\end{corollary}

 \begin{corollary} \label{corollary4}
 Let $f: M \mapsto M$   be a continuous map  on a compact Riemannian manifold $M$. Then, the following conditions are equivalent:

 \begin{itemize}
 \item[(i)  ] \ \ \ The set of pseudo-physical measures is finite.

 \item[(ii) ] \ \    There exists a finite number of (individually) empirically stochastically stable measures, hence physical measures, and the union of their strong basins of statistical attraction covers Lebesgue a.e.

 \end{itemize}
 \end{corollary}

\begin{corollary} \label{corollary5}
 If the set of   pseudo-physical measures is  countable, then there exists countably many  empirically stochastically stable measures, hence physical, and the union of their strong basins of statistical attractions covers Lebesgue a.e.
 \end{corollary}

 \begin{corollary}
 \label{corollaryAllC1ExpandingMapsInTheCircle} For all $C^1$-expanding maps of the circle, all the measures of any empirically stochastically stable set ${\mathcal K}$ satisfy Pesin Entropy Formula  \em (\ref{eqnPesinEntropyFormula}).
 \end{corollary}

 \begin{corollary}
 \label{corollaryC0genericMapsOfTheInteval}
 For $C^0$-generic maps of the interval, the globally empirically stochastically stable set ${\mathcal K}$ of invariant measures includes all the ergodic measures but is meager in the whole space of invariant measures.
 \end{corollary}

 \section{Proof of Theorem \ref{theoremEmpiricStochasticStableMeasure=Physical} and its corollaries.}

\label{sectionProofTheorem1}

 We decompose the proof of Theorem \ref{theoremEmpiricStochasticStableMeasure=Physical} into several lemmas:

\begin{lemma}
\label{lemma1}

For $\epsilon >0$  small enough:

\noindent \em (a) \em  The transformation  $x \in M \mapsto p_{\epsilon} (x, \cdot) \in {\mathcal M}$ is continuous.

\noindent \em (b) \em  The transfer operator ${\mathcal L}_{\epsilon}^*: {\mathcal M} \mapsto {\mathcal M}$ is continuous.

\noindent \em (c) \em The transformation $x \in M \mapsto \sigma_{\epsilon, n,x} \in {\mathcal M}$ is continuous.

\noindent \em (d) \em  ${\lim}^*_{\epsilon \rightarrow 0^+} p_{\epsilon}(x, \cdot) = \delta_{f(x)}$ uniformly on $ M$.

\noindent \em (e) \em  ${\lim}^*_{\epsilon \rightarrow 0^+} {{\mathcal L}^*_{\epsilon}}^n \delta_x = \delta_{f^n(x)} $ uniformly on $ M$.

\noindent \em (f) \em  ${\lim}^*_{\epsilon \rightarrow 0^+}   \sigma_{\epsilon, n, x}  = \sigma_{n, x} $ uniformly on $ M$.

\end{lemma}
\begin{proof}
\smartqed
 \noindent   (a) :    It is immediate from the construction of the probability measure $p_{\epsilon}(x, \cdot)$ by equality (\ref{equationp_epsilon}), and taking into account that the Lebesgue measure restricted to a ball of radius $\epsilon$ depends continuously on the center of the ball.

  \noindent (b):  Take a convergent sequence $\{\mu_i\}_{i \in \mathbb{N}} \subset {\mathcal M}$  and denote $\mu = \lim_i^* \mu_i$. For any continuous function  $\varphi: M \mapsto M$, we have
  \begin{equation}
  \label{eqn03}
  \int \varphi d {\mathcal L}_{\epsilon}^* \mu_i = \int {\mathcal L}_{\epsilon} \varphi \, d \mu_i .\end{equation}
  Since $(\mathcal L_{\epsilon} \varphi) (x) = \int \varphi(y) p_{\epsilon}(x, dy)$ and $p_{\epsilon}(x, \cdot)$ depends continuously on $x$, we deduce that ${\mathcal L}_{\epsilon} \varphi$ is a continuous function. So, from (\ref{eqn03}) and the definition of the weak$^*$ topology in ${\mathcal M}$, we obtain:
 $$ \lim_{i \rightarrow + \infty} \int \varphi d {\mathcal L}_{\epsilon}^* \mu_i = \lim_{i \rightarrow + \infty}\int {\mathcal L}_{\epsilon} \varphi \, d \mu_i =  \int {\mathcal L}_{\epsilon} \varphi \, d \mu =  \int \varphi d {\mathcal L}_{\epsilon}^* \mu .  $$
 We conclude that $\lim_i^*{\mathcal L}_{\epsilon}^* \mu_i = {\mathcal L}_{\epsilon}^* \mu, $ hence ${\mathcal L}^*_{\epsilon}$ is a continuous operator on ${\mathcal M}$.

  \noindent (c):   Since the composition of continuous operators is continuous, we have that ${{\mathcal L}_{\epsilon}^*}^j: {\mathcal M} \mapsto {\mathcal M}$ is continuous for each fixed $j \in {\mathbb{N}^+}$. Besides, it is immediate to check that the transformation $x \in M \mapsto \delta_x \in {\mathcal M}$ is continuous. Thus, also the transformation $x \in M \mapsto {{\mathcal L}_{\epsilon}^*}^j \delta_x \in {\mathcal M} $  is continuous. We conclude that, for fixed $\epsilon >0$ and fixed $n \in \mathbb{N}^+$, the transformation
  $$x \in M \ \mapsto \ \sigma_{\epsilon, n, x} = \frac{1}{n} \sum_{j= 1}^n {{\mathcal L}_{\epsilon}^*}^j \delta_x \in {\mathcal M}$$
  is continuous.

  \noindent (d):   For any given $\rho>0$ we shall find $\epsilon_0 >0$ (independent on $x \in M$) such that, $\mbox{dist}^*(p_{\epsilon} (x, \cdot), \ \delta_{f(x)}) <\rho$ for all $0 <\epsilon <\epsilon_0$ and for all $x \in M$. For any metric $\mbox{dist}^*$ that endows the weak$^*$ topology in ${\mathcal M}$, the inequality $\mbox{dist}^*(p_{\epsilon} (x, \cdot), \ \delta_{f(x)}) <\rho$ holds, if and only if, for a \em finite \em number (which depends on $\rho$ and on the metric) of continuous functions $\varphi: M \mapsto \mathbb{C}$, the difference $|\int \varphi(y) \, p_{\epsilon}(x, dy) - \varphi (f(x)) |$ is smaller than a certain $\epsilon'>0$ (which depends on $\rho$ and on the metric).  Let us fix such a continuous function $\varphi$. Since $M$ is compact, $\varphi$ is uniformly continuous on $M$. Thus, for any $\epsilon' >0$ there exists $\epsilon_0$ such that, if $\mbox{dist}(y_1, y_2) <\epsilon \leq \epsilon_0$, then
  $|\varphi(y_1) - \varphi(y_2) | < \epsilon'.$ Since $p_{\epsilon}(x, \cdot)$ is supported on the ball $B_{\epsilon}(f(x))$, we deduce:
  $$\Big |\int \varphi(y) p_{\epsilon}(x, dy) - \varphi(f(x))\Big| \leq \int \big| \varphi(y) - \varphi (f(x)\big | \,  p_{\epsilon}(x, dy) \leq \epsilon',$$
  because $\mbox{dist} (y, f(x)) < \epsilon \leq \epsilon_0$ for $p_{\epsilon}(x, \cdot)$- a.e. $y \in M$.

  Since $\epsilon_0$ does not depend on $x$,  we have proved that ${\lim}^*_{\epsilon \rightarrow 0^+} p_{\epsilon}(x, \cdot) = \delta_{f(x)}$ uniformly for all $x \in M$.

\noindent (e):  Let us prove that $\lim_{\epsilon \rightarrow 0^+} {{\mathcal L}_{\epsilon}^*}^n \delta_x = \delta_{f^n(x)}$ uniformly on $x \in M$. By induction on $n \in \mathbb{N}^+$:

If $n= 1$,  for any continuous function $\varphi: M \mapsto {\mathbb{C}}$ we compute the following  integral
$$\int \varphi \, d{{\mathcal L}_{\epsilon}^*}  \delta_x = \int ({{\mathcal L}_{\epsilon} }\varphi) \, d \delta_x = ({{\mathcal L}_{\epsilon} }\varphi) (x) = \int \varphi(y) \, p_{\epsilon}(x, dy).  $$
From the unicity of the probability measure of Riesz Representation Theorem, we obtain ${{\mathcal L}_{\epsilon}^*}  \delta_x = p_{\epsilon}(x, \cdot)$. Applying part d), we conclude
$${\lim}^*_{\epsilon  \rightarrow 0^+}{{\mathcal L}_{\epsilon}^*}  \delta_x  = {\lim}^*_{\epsilon  \rightarrow 0^+}p_{\epsilon}(x, \cdot) = \delta_{f(x)}, \mbox{ uniformly on } x \in M.$$

Now, assume that, for some $n \in \mathbb{N}^+$, the following assertion holds: \begin{equation} \label{eqn04}{\lim}^*_{\epsilon  \rightarrow 0^+}{{\mathcal L}_{\epsilon}^*}^n  \delta_x   = \delta_{f^n(x)}, \mbox{ uniformly on } x \in M.\end{equation}
Let us prove the same assertion for $n+1$, instead of $n$:
Fix a continuous function $\varphi: M \mapsto \mathbb{C}$. As proved in part d), for any $\epsilon' >0$, there exists $\epsilon_0 >0$ (independent on $x \in M$) such that
$$|{\mathcal L}_{\epsilon} \varphi) (x) - \varphi(f(x))| = |\int \varphi(y) p_{\epsilon}(x, dy) - \varphi(f(x)) | < \frac{\epsilon'}{2} \ \ \ \forall \ 0 <\epsilon \leq \epsilon_0, \ \ \forall \ x \in M.$$
Thus

$$\Big | \int \varphi \, d {{\mathcal L}_{\epsilon}^*}^{n+1} \delta_x - \int (\varphi \circ f) \, d  {{\mathcal L}_{\epsilon}^*}^{n } \delta_x     \Big| =  \Big | \int ({\mathcal L}_{\epsilon}  \varphi) \, d {{\mathcal L}_{\epsilon}^*}^{n} \delta_x - \int (\varphi \circ f) \, d  {{\mathcal L}_{\epsilon}^*}^{n } \delta_x     \Big|     $$ \begin{equation}
\label{eqn05} \leq \int \big|{\mathcal L}_{\epsilon}  \varphi)  -\varphi \circ f \big| \, d  {{\mathcal L}_{\epsilon}^*}^{n } \delta_x     \Big|  <\frac{\epsilon'}{2} \ \ \ \forall \ 0 <\epsilon \leq \epsilon_0, \ \ \forall \ x \in M.
\end{equation}
Besides, the induction assumption (\ref{eqn04}) implies that, if $\epsilon_0$ is chosen small enough, then
for the continuous function $\varphi \circ f$ the following inequality holds:
$$ \Big| \int  (\varphi \circ f ) \, d  {{\mathcal L}_{\epsilon}^*}^{n } \delta_x  -   \varphi ( f^{n+1}(x) )    \Big| = $$ \begin{equation}
\label{eqn06}= \Big| \int  (\varphi \circ f ) \, d  {{\mathcal L}_{\epsilon}^*}^{n } \delta_x  - \int (\varphi \circ f ) \, d \delta_{f^n(x)}   \Big|    <\frac{\epsilon'}{2} \ \ \ \forall \ 0 <\epsilon \leq \epsilon_0, \ \ \forall \ x \in M.
\end{equation}
Joining inequalities (\ref{eqn05}) and (\ref{eqn06}) we deduce that for all $\epsilon' >0$, there exists $\epsilon_0>0$ (independent of $x$) such that
$$\Big | \int \varphi \, d {{\mathcal L}_{\epsilon}^*}^{n+1} \delta_x - \int \varphi  \, d  \delta_{f^{n+1}(x)}   \Big| < \epsilon' \ \ \ \ \forall \ 0 <\epsilon \leq \epsilon_0, \ \ \forall \ x \in M. $$
In other words:
$${\lim}^*_{\epsilon \rightarrow 0^+} {{\mathcal L}_{\epsilon}^*}^{n+1} \delta_x  = \delta_{f^{n+1}(x)} \ \ \mbox{ uniformly on } x \in M,$$
ending the proof of part (e).

\noindent (f):  Since $\sigma_{\epsilon, n, x} = \frac{1}{n} \sum_{j=1}^{n} {{\mathcal L}^*_{\epsilon}}^j \delta_{x}$, applying part (e) to each probability measure ${{\mathcal L}^*_{\epsilon}}^j \delta_{x}$, we deduce that
$${\lim}^*_{\epsilon \rightarrow 0^+} {{\mathcal L}_{\epsilon}^*}^{n+1} \delta_x  = \frac{1}{n} \sum_{j= 1}^n\delta_{f^{j}(x)} = \sigma_{n,x} \ \ \mbox{ uniformly on } x \in M,$$
ending the proof of Lemma \ref{lemma1}.
\qed
\end{proof}

\begin{lemma}
\label{lemma3}
For any probability measure $\mu$ consider the (maybe empty) basin of stochastic stability $\widehat A_{\mu}$ defined by equality \em (\ref{equation01})\em, and the (maybe empty) strong basin of statistical attraction $A_{\mu}$  defined by equality \em (\ref{equation02})\em.

Then, $\widehat A_{\mu}$ and $A_{\mu}$ are measurable sets and coincide. Besides, they satisfy the following equality:
\begin{equation} \label{equation03}\widehat A_{\mu} = A_{\mu} = \bigcap_{k \in \mathbb{N}^+} \bigcup_{N \in \mathbb{N}^+} \bigcap_{n \geq N} C_{n, \ 1/k } (\mu), \end{equation} where, for any   real number $\rho>0$ and any natural number $n \geq 1$,  the set $C_{n, \ \rho}(\mu) $ is defined by
$$ C_{n, \ \rho}(\mu) := \{ x \in M \colon \ \ \mbox{dist}^*(\sigma_{n, x}, \ \mu) < \rho\}.$$
\end{lemma}

\begin{proof}
\smartqed
 From equality (\ref{equation02}), we re-write the strong basin of statistical attraction of $\mu$ as follows:
\begin{equation}
\label{eqn07}
A_{\mu} = \Big \{x \in M: \ \ {\lim}^*_{n \rightarrow + \infty} \sigma_{n, x} = \mu \Big \} = \bigcap_{\rho >0} \bigcup_{N \in \mathbb{N}^+ }\bigcap_{n \geq N} C_{n, \rho}(\mu).\end{equation}

 From equality (\ref{equation01}) we have:
\begin{equation}
\label{eqn08}\widehat A_{\mu} =\bigcap_{\rho>0} \bigcup_{N \in \mathbb{N}^+} \bigcap_{n \geq N} D_{n, \rho}(\mu),\end{equation}
  where $D_{n, \rho}(\mu)$ is  defined by
$$D_{n, \rho}(\mu) : = \bigcup_{\epsilon_0>0} \bigcap_{0 <\epsilon \leq \epsilon_0} \{x \in M\colon \  \ \mbox{dist}^*(\sigma_{\epsilon, n, x}, \ \mu) < \rho\}.$$
 The assertion $\mbox{dist}^*(\sigma_{\epsilon, n, x}, \ \mu) < \rho $ for all $0 <\epsilon \leq \epsilon_0$  implies
 $$\lim_{\epsilon \rightarrow 0^+}\mbox{dist}^*(\sigma_{\epsilon, n, x}, \ \mu) \leq \rho < 2 \rho. $$
 Thus, applying part (f) of Lemma \ref{lemma1}, we deduce that $\mbox{dist}^* (\sigma_{n,x}, \mu) < 2 \rho$ for all $x \in D_{n, \rho}(\mu)$. In other words,
 $$ D_{n, \rho}(\mu)  \subset C_{n, 2 \rho}(\mu),$$
 which, joint with equalities (\ref{eqn07}) and (\ref{eqn08}), implies:
 $$\widehat A_{\mu} \subset A_{\mu}.$$

 To prove the converse inclusion, we apply again part (f) of Lemma \ref{lemma1} to write:
 $$C_{n, \rho}(\mu) = \{x \in X: \ \ \mbox{dist}^* ({\lim}^*_{\epsilon \rightarrow 0^+} \sigma_{\epsilon, n, x}, \mu) < \rho\}$$
 Therefore
 $$\lim_{\epsilon \rightarrow 0^+} \mbox{dist}^* (\sigma_{\epsilon, n, x}, \mu) < \rho\ \ \ \forall \ x \in C_{n, \rho}(\mu).$$
 Thus,
 $$C_{n, \rho} (\mu) \subset \bigcup_{\epsilon_0 >0} \bigcup_{0< \epsilon \leq \epsilon_0} \{ x \in M: \ \mbox{dist}^* (\sigma_{\epsilon, n, x}, \mu) < \rho\} =  D_{n, \rho}(\mu).$$
The above inclusion, joint with equalities (\ref{eqn07}) and (\ref{eqn08}), implies
$$A_{\mu} \subset \widehat A_{\mu}.$$

We have proved that
$$\widehat A_{\mu} = A_{\mu} = \bigcap _{\rho >0} \bigcup_{N \in {\mathbb{N}^+}} \bigcap_{n \geq N} C_{n, \rho}(\mu).$$
Since the   set $C_{n, \rho}(\mu)$  decreases when $\rho$ decreases (with $n$ and $\mu$ fixed),  the family
$$\Big \{ \bigcup_{N \in {\mathbb{N}^+}} \bigcap_{n \geq N} C_{n, \rho}(\mu).    \Big\}_{\rho >0},$$ whose intersection is $A_{\mu}$, is decreasing with $\rho$ decreases. Therefore, its intersection is equal to the intersection of its countable subfamily
$$\Big \{ \bigcup_{N \in {\mathbb{N}^+}} \bigcap_{n \geq N} C_{n, \ 1/k}(\mu).    \Big\}_{k \in \mathbb{N}^+}.$$
We have proved equality (\ref{eqn03}) of Lemma \ref{lemma3}.

Finally, note that the set $C_{n, \ 1/k}(\mu) \subset M$ is open, because $\sigma_{n,x} = (1/n) \sum_{j= 1}^n \delta_{f^j(x)}$ (with fixed $n$) depends continuously on $x$. Since equality (\ref{eqn03}) states that $\widehat A_{\mu} = A_{\mu}$ is the countable intersection of a countable union of a countable intersection of open sets, we conclude that it is a measurable set, ending the proof of Lemma \ref{lemma3}.
\qed
\end{proof}

\begin{lemma}
\label{lemma4} A probability measure $\mu$ is empirically stochastically stable, according to Definition \em \ref{definitionEmpiricStochasticStabilityMeasure}\em, if and only if its basin $\widehat A_{\mu}$ of empiric stability, defined by equality \em (\ref{equation01})\em, has positive Lebesgue measure.
\end{lemma}

\begin{proof}
\smartqed
If $\mu$ is empirically stochastically stable, then from Definition \ref{definitionEmpiricStochasticStabilityMeasure}, there exists a Lebesgue-positive set $\widehat A \subset M$ such that $\widehat A \subset \widehat A_{\mu}$. Hence $m(\widehat A_{\mu}) >0$.

To prove the converse assertion, assume that  $m(\widehat A_{\mu}) = \alpha >0.$
Let us construct a positive Lebesgue set $\widehat A \subset \widehat A_{\mu}$ such that for any $\rho>0$, there exists $N \in \mathbb{N}^+$ (uniform on $x \in \widehat A$), such that for all $n \geq N$ there exists $\epsilon_0 >0$ (uniform on $x \in \widehat A$) satisfying
\begin{equation}
\label{eqntobeproved}
\mbox{dist}^*(\sigma_{\epsilon, n, x}, \mu) < \rho \ \ \ \forall \ 0<\epsilon \leq \epsilon_0, \ \ \forall \ x \in \widehat A \ \mbox{ (to be proved)}.\end{equation}

Applying Lemma \ref{lemma3} we have
 $$\widehat A_{\mu} = \bigcap_{k \in \mathbb{N}^+} \bigcup_{N \in \mathbb{N}^+} E_{N, 1/ k}, \ \ \mbox{
where } \ \ E_{N,1/k} := \bigcap_{n \geq N} C_{n, 1/k}(\mu). $$
For fixed $k \in \mathbb{N}^+$ we have   $E_{N+1, 1/ k}\subset E_{N, 1/k}$ for all $N \geq 1$, and
$$\widehat A_{\mu}  =  \bigcup_{N \in \mathbb{N}^+} ( E_{N, 1/ k} \cap \widehat A_{\mu}). \ \ \ \mbox{ Then } \ \
\lim_{N \rightarrow + \infty} m(E_{N, 1/ k} \cap \widehat  A_{\mu}) =  m (\widehat A_{\mu}) = \alpha.$$
Therefore, for each $k \geq 1$ there exists $N(k) \geq 1$ such that  $$\alpha (1- 1/3^k) \leq m(E_{N(k), 1/k}\cap \widehat A_{\mu})  \leq \alpha .$$
We construct
$$\widehat A := \bigcap_{k \in {\mathbb{N}^+}}(E_{N(k), 1/k}\cap \widehat A_{\mu}). $$
We will prove that $\widehat A$ has positive Lebesgue measure and that assertion (\ref{eqntobeproved}) is satisfied uniformly for all $x \in \widehat A$.
First, $$m( \widehat A_{\mu} \setminus \widehat A) =  m (\bigcup_{k\geq1} ( \widehat A_{\mu} \setminus E_{N(k), 1/k} ) \leq \sum_{k= 1}^{+ \infty} (\alpha -m( E_{N(k), 1/k} \cap \widehat A_{\mu})) \leq \sum_{k= 1}^{+ \infty}  \frac{\alpha}{3^k} = \frac{\alpha}{2}, $$
 from where $$ m( \widehat A) = m (\widehat A_{\mu} ) - m(A_{\mu} \setminus \widehat A) \geq \alpha - \frac{\alpha}{2} = \frac{\alpha}{2} >0. $$
 Second,  for all $\rho>0$, there exists a natural number $k \geq 2/\rho$, and a set

 \noindent $B_{N(k), 1/k}     \supset \widehat A$ such that
 $$x \in C_{n, 1/k}(\mu) \ \ \ \forall \ n \geq N(k), \ \ \ \ \forall \ \ \ x \in B_{N(k), 1/k}.  $$
 Therefore,  for all $n \geq N(k)$ (which is independent on $x$) we obtain:
 \begin{equation} \label{eqn09a}\mbox{dist}^*(\sigma_{n, x}, \ \mu) < \frac{1}{k}  \leq \frac{\rho}{2} \ \ \ \ \forall \ x \in \widehat A. \end{equation}
 Finally, applying part (f) of Lemma \ref{lemma1},  for each fixed $n \geq N(k)$ there exists $\epsilon_0>0$ (independent of $x$), such that
 \begin{equation} \label{eqn09b}\mbox{dist}^*(\sigma_{\epsilon, n, x},  \ \sigma_{n, x}) <  \frac{\rho}{2}   \ \ \ \ \forall \ 0 < \epsilon \leq \epsilon_0, \ \ \forall \ x \in \widehat M. \end{equation}
  Inequalities (\ref{eqn09a}) and (\ref{eqn09b})  end the proof of inequality (\ref{eqntobeproved}); hence Lemma \ref{lemma4} is proved.
\qed
\end{proof}

\noindent {\bf End of the proof of Theorem \ref{theoremEmpiricStochasticStableMeasure=Physical}.}

\begin{proof}
\smartqed
From Lemma \ref{lemma4}, $\mu$ is empirically stochastically stable if and only if $m(\widehat A_{\mu}) >0$. From Definition \ref{definitionPhysicalMeasures}, $\mu$ is physical if and only if $m(A_{\mu}) >0$. Applying Lemma \ref{lemma3} we have $\widehat A_{\mu} = A_{\mu}$. We conclude that $\mu$ is empirically stochastically stable if and only if $\mu$ is physical.
\qed
\end{proof}

Before proving Corollary \ref{corollaryGlobalStabilityUniquePhysicalMeasure}, we recall the following theorem taken from \cite{Polaca}:

\begin{theorem}
\label{TheoremCE}

Let $f: M \mapsto M$ be a continuous map on a compact Riemannian manifold $M$.
Then, the set ${\mathcal O}_f$ of pseudo-physical measures for $f$ is nonempty and weak$^*$-compact, and contains $p \omega _x$ for Lebesgue-a.e. $x \in M$.

Moreover,  ${\mathcal O}_f$ is the  minimal nonempty weak$^*$-compact set of probability measures that contains $p \omega_x$ for Lebesgue-a.e. $x \in M$.

\end{theorem}
\begin{proof}
See  \cite[Theorem 1.5]{Polaca}.
\end{proof}

\noindent{\bf  Proof of Corollary \ref{corollaryGlobalStabilityUniquePhysicalMeasure}.}

\begin{proof}
\smartqed
(i) implies (ii):  If $\mu_1$ is globally empirically stable, then by Definition \ref{definitionGloballyEmpStochStabilityMeasure} $m(\widehat A_{\mu_1}) = m(M)$. Applying Theorem \ref{theoremEmpiricStochasticStableMeasure=Physical}, $\mu_1$ is physical. Besides, from Lemma \ref{lemma3}, we know $\widehat A_{\mu_1} = A_{\mu_1}$. Then $m(A_{\mu_1}) = m (M)$. So, there exists $\mu_2 = \mu_1$ that is physical and whose strong basin of statistical attraction has full Lebesgue measure, as wanted.

(ii) implies (iii):  If $\mu_2$ is physical and $m(A_{\mu_2})=m(M)$, then from Definitions  \ref{definitionEmpiricProbab&pomega} and \ref{definitionBasinStatisticalAttraction}, we deduce that the set $\{\mu_2\}$ contains $p\omega_x$ for Lebesgue-a.e. $x \in M$. Besides $\{\mu_2\}$ is nonempty and weak$^*$-compact. Hence, applying the last assertion of Theorem \ref{TheoremCE}, we deduce that $\{\mu_2\}$ is the whole set ${\mathcal O}_f$ of pseudo physical measures for $f$. In other words, there exists a unique measure $\mu_3 = \mu_2$ that is pseudo-physical, as wanted.

(iii) implies (i):  If there exists a unique measure $\mu_3$ that is pseudo-physical for $f$, then, applying Theorem \ref{TheoremCE} we know that that the set $\{\mu_3\}$ contains $p\omega_x$ for Lebesgue-a.e. $x \in M$. From Definitions \ref{definitionEmpiricProbab&pomega} and \ref{definitionBasinStatisticalAttraction}, we deduce that the strong basin $A_{\mu_3}$ of statistical attraction of $\mu_3$ has full Lebesgue measure. Then, $\mu_3$ is physical, and applying Theorem \ref{theoremEmpiricStochasticStableMeasure=Physical} $\mu_3$ is empirically stochastically stable. Besides, from   Lemma \ref{lemma3}, we obtain that the basin $\widehat A_{\mu_3}$ of empiric stochastic stability of $\mu_3$ coincides with $A_{\mu_3}$; hence it has full Lebesgue measure. From Definition \ref{definitionGloballyEmpStochStabilityMeasure} we conclude that there exists a measure $\mu_1 = \mu_3$ that is globally empirically stochastically stable, as wanted.

We have proved that (i), (ii) and (iii) are equivalent conditions. Besides, we have proved that if these conditions holds, the three measures $\mu_1$, $\mu_2 $ and $\mu_3$ coincide. This ends the proof of Corollary \ref{corollaryGlobalStabilityUniquePhysicalMeasure}.
\qed
\end{proof}

\noindent{\bf  Proof of Corollary \ref{corollaryC1&C2ExpandingMapsInTheCircle}.}

\begin{proof}
\smartqed
On the one hand, a classical theorem by Ruelle states that any $C^2$ expanding map $f$ of the circle $S^1$ has a unique  invariant measure $\mu$ that is ergodic and absolutely continuous with respect to the Lebesgue measure. Thus, from Pesin's Theory \cite{PesinTheory1}, \cite{PesinTheory2}, it is the unique invariant measure that satisfies Pesin Entropy Formula
(\ref{eqnPesinEntropyFormula}).

On the other hand, Qiu \cite{Qiu} has proved that $C^1$ generic transitive and hyperbolic maps have a unique invariant measure $\mu$ that satisfies Pesin Entropy Formula, (nevertheless,     $C^1$-generically $\mu$ is mutually singular with the Lebesgue measure \cite{Avila-Bochi}).  We deduce that, in particular $C^1$ generic expanding maps of $S^1$ have a unique measure $\mu$ satisfying equality (\ref{eqnPesinEntropyFormula}).

Applying the above known results, to prove this corollary we will first prove that for any $C^1$ expanding map $f$, if it exhibits a unique invariant measure $\mu$ that satisfies (\ref{eqnPesinEntropyFormula}), then $\mu$ is the unique empirically stochastically stable measure.
In fact, in \cite{Portugalia} it is proved that any pseudo-physical measure of any $C^1$ expanding map of $S^1$ satisfies Pesin Entropy Formula  (\ref{eqnPesinEntropyFormula}). Hence, we deduce that, for our map $f$, $\mu$ is the unique pseudo-physical measure.  Besides in \cite{Polaca}, it is proved that if the set of pseudo-physical or SRB-like measures is finite, then all the pseudo-physical measures are physical. We deduce that our map $f$ has a unique physical measure $\mu$. Applying Theorem \ref{theoremEmpiricStochasticStableMeasure=Physical}, $\mu$ is the unique empirically stochastically stable measure, as wanted.

Now, to end the proof of this corollary, let us show that the measure $\mu$ that was considered above,  is globally empirically stochastically stable. From Theorem \ref{TheoremCE}, the set ${\mathcal O}_f$ of all the pseudo-physical measures is the minimal weak$^*$-compact set of invariant measures such that $p\omega(x) \subset {\mathcal O}_f$ for Lebesgue-a.e. $x \in S^1$. But, in our case, we have  ${\mathcal O}_f = \{\mu\}$; hence $p\omega(x) = \{\mu\}$ for Lebesgue-a.e. $x \in S^1$. Applying Definition \ref{definitionBasinStatisticalAttraction}, we conclude that the strong basin of statistical attraction $A_{\mu}$ has full Lebesgue measure; and so, by Theorem \ref{theoremEmpiricStochasticStableMeasure=Physical} the basin $\widehat A_{\mu}$ of empirically stochastic stability of $\mu$ covers Lebesgue-a.e. the space; hence $\mu$ is globally empirically stochastically stable.
\qed
\end{proof}

\section{Proof of Theorem \ref{theoremEmpricStochasticStableSetsSubsetPseudoPhysical} and its corollaries.}

\label{sectionProofTheorem2}

For any nonempty weak$^*$-compact set ${\mathcal K}  $ of $f$-invariant measures, recall Definition \ref{definitionBasinWidehatA_K} of the (maybe empty) basin $\widehat A_{\mathcal K} \subset M$ of empiric stochastic stability of ${\mathcal K}$ constructed by equality (\ref{equation11}).

Similarly to Definition \ref{definitionBasinStatisticalAttraction}, in which the strong basin $A_{\mu}$ of statistical attraction of a single measure $\mu$ is constructed, we   define now \em the (maybe empty) strong basin of statistical attraction $A_{\mathcal K} \subset M$ of the set ${\mathcal K} \subset {\mathcal M}$, \em as follows:
\begin{equation}
\label{equation12}
A_{\mathcal K} := \{x \in M, \ \ p \omega_x \subset {\mathcal K}\},
\end{equation}
where  $p \omega_x$  is the   $p$-omega limit set (limit set in the space ${\mathcal M}$ of probabilities) for the empiric probabilities along the orbit with initial state in $x \in M$ (recall Definition \ref{definitionEmpiricProbab&pomega}).

 We will prove the following property of the basins $\widehat A_{\mathcal K}  $ and $A_{\mathcal K}$:

 \begin{lemma}
 \label{lemma3bis}

 For any nonempty weak$^*$-compact set ${\mathcal K}$ in the space ${\mathcal M}$ of probability measures, the basins
 $\widehat A_{\mathcal K} \subset M$ and $A_{\mathcal K} \subset M$, defined by equalities \em (\ref{equation11}) \em and \em (\ref{equation12}) \em respectively, are measurable sets and coincide. Moreover
 $$\widehat A_{\mathcal K} = \widehat A_{\mu} = \bigcap _{k \in {\mathbb{N}^+}} \bigcup_{N \in {\mathbb{N}^+}} \bigcap_{n \geq N} C_{n, 1/k}({\mathcal K}), $$
 where, for all $\rho >0$ the set $C_{n, \rho}({\mathcal K}) \subset M$ is defined by
 $$ C_{n, \rho}({\mathcal K})  = \{x \in M \colon \ \  \mbox{dist}^*(\sigma_{n, x}, \ {\mathcal K}) < \rho\}.$$
 \end{lemma}
\begin{proof}
\smartqed
Repeat the proof of Lemma \ref{lemma3}, with the set ${\mathcal K}$ instead of the single measure $\mu$, and using equalities (\ref{equation11}) and (\ref{equation12}), instead of (\ref{equation01}) and (\ref{equation02}) respectively.
\qed
\end{proof}

\begin{lemma}
\label{lemma4bis}
The set ${\mathcal O}_f$ of all pseudo-physical measures is globally empirically stochastically stable.
\end{lemma}
\begin{proof}
\smartqed
From Theorem \ref{TheoremCE}, $p\omega_x \subset {\mathcal O}_f$ for Lebesgue-a.e. $x \in M$. Thus, the strong basin of statistical attraction $ A_{{\mathcal O}_f}$ of ${\mathcal O}_f$, defined by equality (\ref{equation12}), has full Lebesegue measure. After Lemma \ref{lemma3bis}, the basin $\widehat A_{{\mathcal O}_f}$ of empiric stochastic stability of  ${\mathcal O}_f$, has full Lebesgue measure. Therefore, if we prove that ${\mathcal O}_f$ is empirically stochastically stable, it must be  globally so.

We now repeat the proof of Lemma \ref{lemma4}, using ${\mathcal O}_f$ instead of a single measure $\mu$, to construct a Lebesgue-positive set $\widehat A \subset M$ such that, for all $\rho >0$ and for all $n$ large enough, there exists $\epsilon_0>0$ (independenly of $x \in \widehat A$) such that
$$\mbox{dist}^* (\sigma_{\epsilon, n, x},  \ {\mathcal O}_f) < \rho \ \ \ \forall \  0 <\epsilon \leq \epsilon_0, \ \ \forall \ x \in \widehat A.$$

Thus, ${\mathcal O}_f $ satisfies condition (a) of Definition \ref{definitionEmpiricStochasticStabilitySetK}, to be empirically stochastically stable. Let us prove that ${\mathcal O}_f $  also satisfies condition (b):

Assume that ${\mathcal K} \subset {\mathcal M}_f$ is nonempty and weak$^*$-compact and ${\widehat A}_{{\mathcal O}_f }\subset {\widehat A}_{\mathcal K}$ Lebesgue-a.e. We shall prove that ${\mathcal O}_f \subset {\mathcal K}$. Arguing by contradiction, assume that there exists a probability measure $\nu \in {\mathcal O}_f \setminus {\mathcal K}$. Choose \begin{equation}
 \label{eqn14}  0 <\rho < \frac{\mbox{dist}^* (\nu, \ {\mathcal K})}{2}
\end{equation}

On the one hand, since $\nu$ is pseudo-physical, applying Definitions \ref{definitionBasinStatisticalAttraction} and \ref{definitionPhysicalMeasures},   the $\rho$-weak basin $A^{\rho}_{\nu}$ of statistical attraction of $\nu$ has positive Lebesgue measure. In brief:
\begin{equation}
 \label{eqn15}
m (\{x \in M: \liminf_{n \rightarrow + \infty} \mbox{dist}^*(\sigma_{n,x}, \ \nu) < \rho\}) >0.  \end{equation} From inequalities (\ref{eqn14}) and (\ref{eqn15}), and applying equality (\ref{equation12}), we deduce that
\begin{equation}
 \label{eqn16} m (\{x \in M\colon \ p\omega_x \not \subset {\mathcal K}\}) >0, \ \ \ m(A_{\mathcal K}) < m (M).\end{equation}

  On the other hand, applying Lemma \ref{lemma3bis} and the hypothesis ${\widehat A}_{{\mathcal O}_f }\subset {\widehat A}_{\mathcal K}$ Lebesgue-a.e., we deduce $$A_{{\mathcal O}_f} \subset A_{\mathcal K}  \mbox{ Lebesgue a.e.}.$$
Applying Theorem \ref{TheoremCE} and equality (\ref{equation12}), we have
$$m(A_{{\mathcal O}_f}) = m (M), \ \ \mbox{ from where we deduce } m(A_{{\mathcal K}}) = m (M),$$
contradicting the inequality at right in (\ref{eqn16}).

We have proved that  ${\mathcal O}_f \subset {\mathcal K}$. Thus ${\mathcal O}_f$ satisfies condition (b) of Definition \ref{definitionEmpiricStochasticStabilitySetK}, ending the proof of Lemma \ref{lemma4bis}.
\qed
\end{proof}

\noindent{\bf End of the proof of Theorem \ref{theoremEmpricStochasticStableSetsSubsetPseudoPhysical}}

\begin{proof}
\smartqed
We denote by ${\mathcal O}_f$ the set of all pseudo-physical measures.

\noindent (a) Let ${\mathcal K} \subset {\mathcal M}_f$ be empirically stochastically stable, according to Definition \ref{definitionEmpiricStochasticStabilitySetK}. We shall prove that ${\mathcal K} \subset {\mathcal O}_f$. Assume by contradiction that there exists $\nu \in {\mathcal K} \setminus {\mathcal O}_f$. So, $\nu$ is not pseudo-physical, and applying Definition \ref{definitionPhysicalMeasures}, there exists $\rho>0$ such that the $\rho$-weak basin $A_{\nu}^\rho$ of statistical attraction of $\nu$ has zero Lebesgue measure. In brief, after Definition \ref{definitionBasinStatisticalAttraction}, we have
$$
m(\{ x \in M\colon \ \ \mbox{dist}^*(p\omega_x, \ \nu) <\rho  \})= 0, $$ from where we deduce that \begin{equation}
\label{eqn17} p\omega_x \ \ \subset \ \ {\mathcal M}_f \setminus {\mathcal B}_{\rho}(\nu)  \ \ \ \mbox{ Lebesgue-a.e. } x \in M,\end{equation}
where ${\mathcal B}_{\rho}(\nu)$ is the open ball in the space ${\mathcal M}$ of probability measures, with center at $\nu$ and radius $\rho$.

Applying Lemma \ref{lemma3bis} and equality (\ref{equation12}) we have $$\widehat A_{\mathcal K} = A_{\mathcal K} = \{x \in X \colon \ \ p\omega_x \subset {\mathcal K}\}. $$
Joining with  assertion  (\ref{eqn17}), we deduce that
 $A_{\mathcal K} \subset A_{{\mathcal K} \setminus {\mathcal B}_{\rho}(\nu)} \ \ \mbox{Lebesgue-a.e.}$; and applying again Lemma \ref{lemma3bis} we deduce:
 $$ \widehat  A_{\mathcal K} \subset \widehat A_{{\mathcal K} \setminus {\mathcal B}_{\rho}(\nu)} \ \ \mbox{Lebesgue-a.e.}$$
But, by hypothesis ${\mathcal K}$ is empirically stochastically stable. Thus, it satisfies condition (b) of Definition \ref{definitionEmpiricStochasticStabilitySetK}. We conclude that ${\mathcal K} \ \ \subset \ \ {\mathcal K} \setminus {\mathcal B}_{\rho}(\nu)$, which is a contradiction, ending the proof of part (a) of Theorem \ref{theoremEmpricStochasticStableSetsSubsetPseudoPhysical}.

\vspace{.2cm}

\noindent (b) According to Lemma \ref{lemma4bis}, if ${\mathcal K} = {\mathcal O}_f$, then ${\mathcal K}$ is globally empirically stochastically stable. Now, let us prove the converse assertion. Assume that ${\mathcal K} $ is globally empirically stochastically stable. We shall prove that ${\mathcal K} = {\mathcal O}_f$. Applying part (a) of Theorem \ref{theoremEmpricStochasticStableSetsSubsetPseudoPhysical}, we know that ${\mathcal K} \subset {\mathcal O}_f$. So, it is enough to prove now that ${\mathcal O}_f \subset {\mathcal K}$.

By hypothesis $m(\widehat A_{\mathcal K}) = m(M)$. From Lemma \ref{lemma3bis} we have $\widehat A_{\mathcal K} =   A_{\mathcal K})$. We deduce that $m(A {\mathcal K}) = m(M)$. From this latter assertion and equality (\ref{equation12}), we obtain
$$p\omega_x \subset {\mathcal K} \ \ \mbox{for Lebesgue-a.e. } x \in M.$$
Finally, we apply the last assertion of Theorem \ref{TheoremCE} to conclude that ${\mathcal O}_f \subset {\mathcal K}$, as wanted. This ends the proof of Theorem \ref{theoremEmpricStochasticStableSetsSubsetPseudoPhysical}.
\qed
\end{proof}

\noindent{\bf Proof of Corollary \ref{corollary2}.}

\begin{proof}
\smartqed

This corollary is immediate after Theorem \ref{theoremEmpricStochasticStableSetsSubsetPseudoPhysical} and Lemma \ref{lemma4bis}. In fact,  Lemma \ref{lemma4bis} states that the set ${\mathcal O}_f$, which is composed by all the pseudo-physical measures, is globally empirically stochastically stable. And part (b) of Theorem \ref{theoremEmpricStochasticStableSetsSubsetPseudoPhysical}, states that ${\mathcal O}_f$ is the unique set of $f$-invariant measures that is globally empirically stochastically stable.
\qed
\end{proof}

Before proving Corollaries \ref{corollary3}, \ref{corollary4} and \ref{corollary5}, we recall the following known result:

\begin{theorem}

\label{theoremCE-pwconnected}
For all $x \in M$ the $p$-omega limit set $p\omega_x$ has the following property:

For any pair of measures $\mu_0, \mu_1 \in p\omega_x$ and for every real number $0 \leq \lambda \leq 1$ there exists a measure $\mu_{\lambda}$ such that $\mbox{dist}^*(\mu_0, \mu_{\lambda}) = \lambda \mbox{dist}^*(\mu_0, \mu_1)$.
\end{theorem}
\begin{proof}
See \cite[Theorem 2.1]{Polaca}.
\end{proof}

\noindent{\bf Proof of Corollary \ref{corollary3}.}

\begin{proof}
\smartqed
Assume that $\mu$ is pseudo-physical and isolated in the set ${\mathcal O}_f$ of all pseudo-physical  measures. Then, there exists $\rho >0$ such that: \begin{equation}
\label{eqn18}
\mbox{ if } \nu \in {\mathcal O}_f \mbox{ and } \mbox{dist}^*(\nu, \mu) < \rho, \ \ \mbox{ then } \ \ \nu= \mu.\end{equation}
Since $\mu $ is pseudo-physical, from Definition \ref{definitionPhysicalMeasures} we know that the $\rho$-weak basin $A^{\rho}_{\mu}$ of statistical attraction of $\mu$ has positive Lebesgue measure. From Definition \ref{definitionBasinStatisticalAttraction} we deduce that
\begin{equation}
\label{eqn19}m(A^{\rho}_{\mu} = m (\{ x \in M \colon \ \ \mbox{dist}
^*(p \omega_x, \mu) < \rho  \})>0.\end{equation}

Applying Theorem \ref{TheoremCE}, we know that $p \omega_x \subset {\mathcal O}_f$ for Lebesgue-a.e. $x \in M$.
Joining the latter assertion with (\ref{eqn18}) and (\ref{eqn19})  we deduce that $$\{\mu\}  = p \omega_x \bigcap {\mathcal B}_{\rho} {\mu} \mbox{ for Lebesgue-a.e. }  x \in A^{\rho} _{\mu},$$
where ${\mathcal B}_{\rho} {\mu}$ is the ball in the space of probability measures, with center at $\mu$ and radius $\rho$.

Besides, from Theorem \ref{theoremCE-pwconnected} we deduce that   $p \omega_x = \{\mu\}$ for Lebesgue-a.e. $x \in A^{\rho} _{\mu} $, hence for a Lebesgue-positive set of points $x \in M$. Applying Definition \ref{definitionPhysicalMeasures}, we conclude that the given pseudo-physical measure $\mu$ is physical; hence, from Theorem \ref{theoremEmpiricStochasticStableMeasure=Physical}, $\mu$ is empirically stochastically stable.
\qed
\end{proof}

 \noindent{\bf Proof of Corollary \ref{corollary4}.}

 \begin{proof}
 \smartqed
 (i) implies (ii): If the set $\mathcal O_f$ of pseudo-physical measures is finite, then all the pseudo-physical are physical due to Corollary \ref{corollary3}. Then, applying Theorem \ref{theoremEmpiricStochasticStableMeasure=Physical}, all of them are (individually) empirically stochastically stable. Besides the union of their strong basins of statistical attraction has full Lebesgue measure: In fact,  applying Definition \ref{definitionBasinStatisticalAttraction} and equality (\ref{equation12}),  that union is the  set $A_{\mathcal O_f}$; and, due to Theorem \ref{TheoremCE}, the set $A_{{\mathcal O}_f}$  has full Lebesgue measure. So, assertion (ii) is proved.

 \noindent (ii) implies (i): Assume that there exists a finite number  $r \geq 1$ of empirically stochastically stable measures $\mu_1, \mu_2, \ldots, \mu_r$ (hence, physical measures, due to Theorem \ref{theoremEmpiricStochasticStableMeasure=Physical}). Assume also that the strong basins $A_{\mu_i}$ of statistical attraction  have an union $\bigcup_{i=1}^r A_{\mu_i}$ that covers Lebesgue-a.e.. Applying    Definition \ref{definitionBasinStatisticalAttraction} and equality (\ref{equation12}), we deduce that $A_{\{\mu_1, \ldots, \mu_r\}} = \bigcup_{i=1}^r A_{\mu_i}$ has full Lebesgue measure. So, from the last assertion of Theorem \ref{TheoremCE}, ${\mathcal O}_f \subset \{\mu_1, \ldots, \mu_r\}$. In other words, the set ${\mathcal O}_f$ of pseudo-physical measures is finite, proving assertion (i).
 \qed
 \end{proof}

 \noindent{\bf Proof of Corollary \ref{corollary5}.}

 \begin{proof}
 \smartqed
 If the set ${\mathcal O}_f$ is finite, then we apply Corollary (\ref{corollary4}) to deduce that there exists a finite number of empirically stochastically stable measures, hence physical, and that the union of their strong basins of statistical attraction has full Lebesgue measure.

 Now let us consider the case for which, by hypothesis, the set ${\mathcal O}_f$ of pseudo-physical measures
   is  countably infinite. In brief: ${\mathcal O}_f = \{\mu_i\}_{i \in \mathbb{N}}$.

    Applying Theorems \ref{TheoremCE},  the $p$-omega limit sets $p\omega_x$ are contained in ${\mathcal O}_f $ for Lebesgue-a.e. $x \in M$. But, from Theorem \ref{theoremCE-pwconnected} we know that $p\omega_x$ is either a single measure or uncountably infinite. Since it is contained in the countable set ${\mathcal O}_f$, we deduce the $p\omega_x $ is composed by a single measure of ${\mathcal O}_f$ for Lebesgue-a.e. $x \in M$. Now, recalling Definition \ref{definitionBasinStatisticalAttraction} and equality (\ref{equation12}), we deduce that
  $$A_{\mathcal O_{f}} = \bigcup_{i= 1}^{+ \infty} A_{\mu_i}, \ \ \ \ \ \ \ \ \ \  \ \sum_{i= 1}^{+ \infty} m(A_{\mu_i}) = m(M).$$
Therefore, there exists finitely many or  countable infinitely many pseudo-physical measures $\mu_{i_n} \colon 1 \leq n \leq r \in \mathbb{N}^+ \cup \{+ \infty\}$ such that \begin{equation} \label{eqn20} \mu(A_{\mu_{i_n}}) >0  \ \ \forall \ 1 \leq n \leq r,  \ \ \ \ \ \ \ \ \ \ \ \ \  \ \ \sum_{n= 1}^{r} m(A_{\mu_n}) = m(M).  \end{equation}
From Definition \ref{definitionPhysicalMeasures}, each measure $\mu_{i_n}$ is physical; hence empirically stochastically stable due to Theorem \ref{theoremEmpiricStochasticStableMeasure=Physical}). Besides, from equality at right in (\ref{eqn20}), we deduce that the union   $\bigcup_{n= 1}^r A_{\mu_{i_n}}$   has full Lebesgue measure, as wanted.

Finally, to end the proof of Corollary \ref{corollary5}, let us show that the set  $\{\mu_{i_n} \colon   1 \leq n \leq r$ of physical measures above constructed, can not be   finite. In brief, let us prove that $r= + \infty$. In fact, if there existed a finite number $r \in {\mathcal N}^+$ of physical measures  whose basins of statistical attraction have an union with full Lebesgue measure, then, we would apply Corollary \ref{corollary4} and deduce that the set ${\mathcal O}_f$ of pseudo-physical measures is finite. But in our case, by hypothesis, ${\mathcal O}_f$ is countably infinite, ending the proof of Corollary \ref{corollary5}.

 \qed
 \end{proof}

\noindent{\bf  Proof of Corollary \ref{corollaryAllC1ExpandingMapsInTheCircle}.}

\begin{proof}
\smartqed
From part a) of Theorem \ref{theoremEmpricStochasticStableSetsSubsetPseudoPhysical} we know that all the measures of any empirically stochastically stable set ${\mathcal K} \subset {\mathcal M}_f$ is pseudo-physical. Besides, in \cite{Portugalia}   it is proved that, for any $C^1$ expanding map $f$ of the circle, any pseudo-physical or SRB-like measure  satisfies Pesin Entropy Formula (\ref{eqnPesinEntropyFormula}). We conclude that all the measures of ${\mathcal K}$ satisfy this formula.
\qed
\end{proof}

\noindent{\bf  Proof of Corollary \ref{corollaryC0genericMapsOfTheInteval}.}

\begin{proof}
\smartqed
From part b) of  Theorem \ref{theoremEmpricStochasticStableSetsSubsetPseudoPhysical} we know that the globally empirically stochastically stable set ${\mathcal K}$ coincides with the set ${\mathcal O}_f$ of pseudo-physical measures. Besides, in \cite{Cat-Troubetzkoy}  it is proved that, for $C^0$-generic maps $f$ of the interval,  any ergodic measure belongs to ${\mathcal O}_f$ but, nevertheless  ${\mathcal O}_f$ is a weak$^*$-closed with empty interior in the space ${\mathcal M}_f$ of invariant  measures. We conclude that all ergodic measures belong to the globally empirically stochastically stable set ${\mathcal K}$ and that this set of invariant measures is meager in ${\mathcal M}_f$, as wanted.
\qed
\end{proof}

\newpage

 \begin{acknowledgement}
We thank  IMPA of Rio de Janeiro (Brazil), and the organizing and scientific committees of the conference \lq\lq New Trends on One-Dimensional Dynamics.  Celebrating the 70th. Anniversary of Welington de Melo\rq\rq, held at IMPA in 2015. We also thank the  financial support by the Program MATHAMSUD (France, Chile and Uruguay)  and ANII (Uruguay) of the Project \lq\lq PhySeCo\rq\rq, and by CSIC- Universidad de la Rep\'{u}blica (Uruguay)  of the Group  Project N. 618,  \lq\lq Dynamical Systems\rq\rq.
\end{acknowledgement}
%


%

\end{document}